\documentclass[pdftex,a4paper,12pt]{article}
\usepackage{amssymb,amsmath,amsfonts,nomencl,mathrsfs}
\usepackage{amsthm}  
\usepackage[arrow, matrix, curve]{xy}
\usepackage[latin1]{inputenc}

\newcommand{\IB}{\mathcal{B}}
\newcommand{\IC}{\mathbb{C}}
\newcommand{\IR}{\mathbb{R}}
\newcommand{\IMM}{\mathscr{M}}

\newcommand{\ITT}{\mathscr{T}}

\newcommand{\ILL}{\mathscr{L}}

\newcommand{\IFF}{\mathscr{F}}

\newcommand{\IRR}{\mathscr{R}}

\newcommand{\ICC}{\mathscr{C}}
\newcommand{\IN}{\mathbb{N}}

\newcommand{\IK}{\mathbb{K}}

\newcommand*{\longhookrightarrow}%
               {\ensuremath{\lhook\joinrel\relbar\joinrel\rightarrow}}

\newcommand{\Id}{{\rm d}}

\newcommand{\f}{\frac}
\newcommand{\nn}{\nonumber}

\theoremstyle{plain}            
\newtheorem{theorem}{theorem}[section]
\newtheorem{Lemma}[theorem]{Lemma}
\newtheorem{Corollary}[theorem]{Corollary}
\newtheorem{Theorem}[theorem]{Theorem}
\newtheorem{Proposition}[theorem]{Proposition}

\theoremstyle{definition}       
\newtheorem{Definition}[theorem]{Definition}

\newtheorem{Remark}[theorem]{Remark}

\allowdisplaybreaks[2]

\newcounter{myenumi}

\makeatletter
\def\@fnsymbol#1{\ifcase#1\or a\or b\or c\or
   d\or e\or f\or g\or h\or
    i\else\@ctrerr\fi}
\makeatother
\begin{document}

 \begin{titlepage}
   \renewcommand{\thefootnote}{\alph{footnote}}

\title{ Functions with bounded variation on a class of Riemannian manifolds with Ricci 
curvature unbounded from below}
  \author{Batu G\"uneysu\footnote{Institut f\"ur Mathematik,
   Humboldt-Universit\"at zu Berlin; e-mail: gueneysu@math.hu-berlin.de}\\
Diego Pallara\footnote{Dipartimento di Matematica e Fisica \emph{Ennio De Giorgi},
Universit\`a del Salento, Lecce; e-mail: diego.pallara@unisalento.it}
}

\end{titlepage}
\date{}
\maketitle 

\begin{abstract} 
After establishing some new global facts (like a measure theoretic structure theorem
and a new invariant characterization of Sobolev functions) about complex-valued functions with bounded variation on arbitrary
noncompact Riemannian 
manifolds, we extend results of Miranda/the second author/Paronetto/Preunkert 
and of Carbonaro/Mauceri on the heat se\-migroup characterization of the variation of 
$\mathsf{L}^1$-functions to a class of Riemannian manifolds with possibly unbounded
from below Ricci curvature. 
\end{abstract}

\section{Introduction}

In the last decades, the class of integrable functions with bounded variation
($\mathsf{BV}$) has proven to be 
very well suited for the formulation of geometric variation problems in the
\emph{Euclidean} $\IR^m$. The 
essential advantage of $\mathsf{BV}$ in this context when compared with the smaller
Sobolev class 
$\mathsf{W}^{1,1}$ is the possibility of allowing discontinuities along
hypersurfaces. For example, the 
class of characteristic functions of Caccioppoli sets is a subclass of $\mathsf{BV}$
which turns out to 
be the appropriate class to formulate the isoperimetric problem. We refer the reader
to \cite{ambro} for 
various aspects of $\mathsf{BV}$ functions in Euclidean space.  \\
The definition of the variation of a function is intuitively clear and classical for
the case of one 
variable, $m=1$, while the first 
satisfactory (from the geometric/variational point of view) extension of this
definition to arbitrary 
$m$ which descends to the usual definition if $m=1$ has been given in \cite{giorgi},
where De Giorgi 
defines the variation of $f\in \mathsf{L}^1(\IR^m,\IR)$ to be the well-defined 
quantity 
\[
\lim_{t\to 0+} \int_{\IR^m} |\mathrm{grad} (\mathrm{e}^{t\Delta}f)|(x) \Id x\in
[0,\infty].
\]
The main result of \cite{giorgi} implies   
\begin{align} \label{wsx}
&\mathrm{Var}(f):= \lim_{t\to 0+} \int_{\IR^m} |\mathrm{grad} (
\mathrm{e}^{t\Delta}f)|(x) \Id x
\\ \nn
=&\sup\left\{\left.\int_{\IR^m}   f(x) \mathrm{div} \alpha(x)  \Id x\>\right| 
\alpha\in [\mathsf{C}^{\infty}_{0}(\IR^m,\IR)]^m, \left\|\alpha\right\|_{\infty}\leq 1
\right\},
\end{align}
an equality which turned out to be the starting point for all the above mentioned
Euclidean results. 
Indeed, it follows from \eqref{wsx} that $\mathrm{Var}(f)$ is finite, if and only if
the distributional 
gradient of $f$ is given by integrating against a $\IR^m$-valued Borel measure.\\ 
Noting that all of the data in  (\ref{wsx}) can be defined invariantly using
differential forms and 
the exterior derivative (cf. Definition \ref{vari} below) and in view of the
importance of (\ref{wsx}) 
for the Euclidean case, the aim of this paper is to 
formulate and prove the above result in a very large class of noncompact Riemannian
manifolds, 
where we want to allow complex-valued functions.
\\
Our main result here is Theorem \ref{BV}, which essentially reads as follows:
\vspace{1.3mm}

\emph{Let $M$ be a geodesically complete smooth Riemannian manifold whose Ricci
curvature $\IRR$ admits 
a decomposition $\IRR=\IRR_1-\IRR_2$ into pointwise self-adjoint Borel sections
$\IRR_1,\IRR_2\geq 0$ 
in $\mathrm{End}( \mathrm{T}^* M)$ such that $|\IRR_2|\in\mathcal{K}(M)$, the Kato
class of $M$. 
Then for any $f\in \mathsf{L}^1(M)$, the complex-valued version of the equality in
(\ref{wsx}) holds true. }
\vspace{1.3mm}

As bounded functions are always in the Kato class (see section \ref{haup}), the
above result extends the 
results of \cite{mira,carbo} in the sense that we do not have to assume that $\IRR$
is bounded from below, 
the latter condition being equivalent to $|\IRR_-|\in\mathsf{L}^{\infty}(M)$, where
$\IRR_-$ is the negative 
part of the Ricci curvature ($\IRR_-$ can be defined using the spectral calculus on
the fibers of 
$\mathrm{T}^*M$). In fact, we can drop the latter condition using probabilistic
techniques: it follows 
from Weitzenb\"ock\rq{}s formula that the heat semigroup given by the Laplace
operator on $1$-forms is a 
\emph{generalized Schr\"odinger semigroup} in the sense of \cite{G2} whose potential
term is given by 
$\IRR$. Thus we can combine the results from \cite{G2} on probabilistic formulae for
such semigroups 
together with a new probabilistic estimate on Kato functions (see Lemma \ref{xdd}
below), which turns 
out to be just enough to derive the equality in (\ref{wsx}). As a consequence of
Theorem \ref{BV} we 
can derive $\mathsf{L}^p$-type criteria on the Ricci curvature that imply the
validity of (\ref{wsx}). 
Furthermore, it is also possible to derive a certain stability result of (\ref{wsx})
under a class of 
conformal transformations of the underlying Riemannian structure.\\
Let us remark that the Kato condition has been used for decades in order to deal
with \emph{local} 
singularities of potentials in quantum mechanics \cite{kato2,ai,sg,Bro1,sturm,G6}.
On the other side, 
as far as we know, this paper is the first one where the Kato condition is used in
order to deal with 
possibly \emph{globally} unbounded potential terms in a purely Riemann geometric
setting.\vspace{1mm}

This paper is organized as follows: \vspace{1mm}

\emph{Defining} the variation by the right 
hand\footnote{ We prefer the right hand side instead of the left hand side, as the
former also makes 
for \emph{locally} integrable functions} side of the invariant version of
(\ref{wsx}), we first 
collect several properties of the variation 
that are valid for arbitrary (geodesically complete) Riemannian manifolds in section
\ref{gen}. For example, 
we prove that functions having bounded variation are precisely those whose
distributional 
derivative is a generalized vector measure, a new structure theorem, which produces
the classical structure 
theorem in the Euclidean case, but can become technical for nonparallelizable
manifolds (see Theorem 
\ref{Bvmeasures}). This structure theorem produces also a new global characterization of first order Sobolev functions (see Proposition \ref{so}). Among other results, we have also added further equivalent
characterizations of the 
variation to the latter section, like an approximation result (see Theorem
\ref{appros}), and also an 
invariance result of $\mathsf{BV}$ under quasi-isometric changes of the Riemannian
structure (see 
Corollary \ref{eqw}).  \\
Section \ref{haup} is devoted to the formulation and the proof of our main result
Theorem \ref{BV}, as 
well as the before mentioned $\mathsf{L}^p$-type criteria for the validity of
(\ref{wsx}) (see Corollary 
\ref{cor1}), and the stability result of the latter equality under certain conformal
transformations (see 
Corollary \ref{cor2}). In section \ref{haup}, we also recall the necessary
definitions and facts about 
Brownian motion and the Kato class $\mathcal{K}(M)$ of $M$. \\
Finally, we have collected some abstract facts on vector measures on locally compact
spaces in an appendix.

\section{Setting and general facts on the variation}\label{gen}

Let $M$ denote a smooth connected Riemannian manifold without boundary, with
$\mathrm{vol}(\Id x)$ 
the Riemannian volume measure, and $\mathrm{K}_r(x)$ the open geodesic ball with
radius $r$ around $x$. 
Unless otherwise stated, the underlying fixed Riemannian structure will be ommitted
in the notation. We 
set $m:=\dim M$. \\
If $E\rightarrow M$ is a smooth Hermitian vector 
bundle, then, abusing the notation in the usual way, the Hermitian structure will be
denoted with 
$(\bullet,\bullet)_x$, $x\in M$; moreover, $\left|\bullet\right|_x$ will stand for
the norm and the operator norm corresponding to $(\bullet,\bullet)_x$ on each fiber
$E_x$, 
and $\left\langle \bullet,\bullet \right\rangle$ for the inner product in the
Hilbert space 
$\Gamma_{\mathsf{L}^2}(M,E)$, that is,
\begin{align}
  \label{adj}
  \left\langle f_1,f_2 \right\rangle
  =\int_M (f_1(x),f_2(x))_x \mathrm{vol}(\Id x).
\end{align} 
Furthermore, the norm $\left\|\bullet\right\|_p$ on $\Gamma_{\mathsf{L}^p}(M,E)$ is
given by
\begin{align}
\left\|f\right\|_p=\left(\int_M |f(x)|^p_x \mathrm{vol}(\Id x)\right)^{1/p}
\end{align}
if $p\in [1,\infty)$, and $\left\|f\right\|_{\infty}$ is given by the infimum of all 
$C\geq 0$ such that $|f(x)|_x\leq C$ for a.e. $x \in M$. The corresponding operator
norms on 
the spaces of bounded linear operators $\ILL(\Gamma_{\mathsf{L}^p}(M,E))$ will also
be denoted 
with $\left\|\bullet\right\|_p$. If $\tilde{E} \to M$ is a second bundle as above
and if
\[
  D \colon \Gamma_{\mathsf{C}^{\infty}}(M,E)
  \longrightarrow \Gamma_{\mathsf{C}^{\infty}}(M,\tilde{E}) 
\]
is a linear differential operator, then we denote with $D^{\dagger}$
the formal adjoint of $D$ with respect to~\eqref{adj}.  \\
We will apply the above in the following situation: For any $k=0,\dots,m$ we will
consider the smooth 
Hermitian\footnote{so everything has been complexified} vector bundle of $k$-forms 
$\bigwedge^k \mathrm{T}^* M\to  M$, with
\[
\Omega^k_{\ICC}(M):= \Gamma_{\ICC}\Big(M, \bigwedge^k \mathrm{T}^* M\Big),\>\text{
where 
$\ICC= \mathsf{C}^{\infty}, \mathsf{L}^p$, etc.}
\]
In order to make the notation consistent, we will set $\bigwedge^0 \mathrm{T}^*
M:=M\times \IC$, 
where of course $\ICC(M)=\Omega^0_{\ICC}(M)$ for functions. In particular, all
function spaces are 
spaces of \emph{complex-valued} functions. The subscript \lq\lq{}0\rq\rq{} in
$\ICC_0$ will always 
stand for \lq\lq{}compactly supported elements of $\ICC$\rq\rq{}. Whenever
necessary, we will write 
$\ICC_{\IR}$ for the real-valued elements of $\ICC$. If
\[
\Id_k: \Omega^k_{\mathsf{C}^{\infty}}(M) 
  \longrightarrow \Omega^{k+1}_{\mathsf{C}^{\infty}}(M)
\]
stands for the exterior derivative, then the Laplace-Beltrami operator acting on
$k$-forms on $M$ 
is given as 
\[
-\Delta_k:=\Id^{\dagger}_k\Id_k+\Id_{k-1}\Id^{\dagger}_{k-1}:\Omega^k_{\mathsf{C}^{\infty}}(M) \longrightarrow \Omega^{k}_{\mathsf{C}^{\infty}}(M).
\]
We shall write $\Id:=\Id_0$ for the exterior derivative on functions, so that 
$-\Delta:=-\Delta_0=\Id^{\dagger}\Id$. \vspace{1.2mm}

To make contact with the introduction, we add:

\begin{Remark} If $\alpha\in\Omega^1_{\mathsf{C}^{\infty}}(M)$ and if $X^{(\alpha)}$
is the smooth 
vector field on $M$ corresponding to $\alpha$ by the Riemannian duality on $M$, then
an integration 
by parts shows $\Id^{\dagger}\alpha=-\mathrm{div}X^{(\alpha)}$. Here, the divergence 
$\mathrm{div}(X)\in\mathsf{C}^{\infty}(M)$ of a smooth vector field $X$ on $M$ is
defined in the 
usual way as follows: locally, if 
\[
\text{$X=\sum^m_{j=1} X^j \f{\partial}{\partial x^j}$ with $X^j
\in\mathsf{C}^{\infty}(M)$, 
then $\mathrm{div}X=\sum^m_{j=1}\f{\partial}{\partial x^j} X^j$.}
\]
Furthermore, if $f\in \mathsf{C}^{\infty}(M)$, then $\mathrm{grad} f$ is the smooth
vector field on 
$M$ given by $\mathrm{grad}f=X^{(\Id f)}$, where locally
$$
\Id f= \sum^m_{j=1}\f{\partial f}{\partial x^j}\Id x^j.
$$

\end{Remark}

The Friedrichs realization of 
$-\Delta_k/2$ in $\Omega^k_{\mathsf{L}^{2}}(M)$ will be denoted with $H_k\geq 0$,
where again 
$H:=H_0$ on functions. We will freely use the fact that for any $p\in [1,\infty]$
the strongly 
continuous self-adjoint semigroup of contractions $(\mathrm{e}^{-t H})_{t\geq
0}\subset \ILL(\mathsf{L}^2(M))$ 
uniquely extends to a strongly continuous semi-group of contractions 
$(\mathrm{e}^{-t H})_{t\geq 0}\subset \ILL(\mathsf{L}^p(M))$, and by local elliptic
regularity, 
$\mathrm{e}^{-tH}f$ has a smooth representative $\mathrm{e}^{-tH}f(\bullet)$ for any
$t>0$, 
$f\in \mathsf{L}^p(M)$. \vspace{1.2mm}

\vspace{2mm}

The following definition is a generalization of Definition (1.4) in \cite{mira} to
\emph{complex-valued} 
and \emph{locally} integrable functions:

\begin{Definition}\label{vari} Let $f\in \mathsf{L}^1_{\mathrm{loc}}(M)$. Then the
quantity  
\begin{align}
\mathrm{Var}(f):&=\sup\left\{\left.\left|\int_M  \overline{ f(x)}\Id^{\dagger}
\alpha(x)  
\mathrm{vol}(\Id x)\right|\>\right| \alpha\in\Omega^1_{\mathsf{C}^{\infty}_0}(M), 
\left\|\alpha\right\|_{\infty}\leq 1\right\}\nn\\
&\>\in [0,\infty]
\end{align}
is called the {\it variation} of $f$, and $f$ is said to have {\it bounded
variation} if 
$\mathrm{Var}(f)<\infty$. 
\end{Definition}
Note the following trivial equalities for any $f\in \mathsf{L}^1_{\mathrm{loc}}(M)$,        
\begin{align}
\mathrm{Var}(f)&=\sup\left\{\left.\left|\int_M   f(x)\overline{\Id^{\dagger}
\alpha(x) } 
\mathrm{vol}(\Id x)\right|\>
\right| \alpha\in\Omega^1_{\mathsf{C}^{\infty}_0}(M),
\left\|\alpha\right\|_{\infty}\leq 1\right\}
\nn\\
&=\sup\left\{\left.\left|\int_M f(x)\Id^{\dagger} \alpha(x)  \mathrm{vol}(\Id x)\right|
\>\right| \alpha\in\Omega^1_{\mathsf{C}^{\infty}_0}(M),
\left\|\alpha\right\|_{\infty}\leq 1\right\}
\nn,
\end{align}
and that of course the property $\mathrm{Var}(f)<\infty$ depends very sensitively on
the 
Riemannian structure of $M$ (see Corollary \ref{eqw} below for a certain stability). \\
We will now collect some properties of the variation which are valid without any
assumptions on the 
Riemannian structure of $M$. 

As in the Euclidean case, the variation of a \emph{smooth} function $f$ can be
calculated explicitly from 
the $\mathsf{L}^1$-norm of $\Id f$:

\begin{Proposition}\label{d14} 
For all $f\in \mathsf{C}^{\infty}(M)$ one has $\mathrm{Var}(f)= \left\|\Id
f\right\|_1$.
\end{Proposition}

\begin{proof} Here, $\mathrm{Var}(f)\leq \cdots$ is clear from 
\[
\int_M  \overline{ f(x)}\Id^{\dagger} \alpha(x)  \mathrm{vol}(\Id x)=
\int_M  (\Id f(x), \alpha(x))_x  \mathrm{vol}(\Id x)\>\>
\text{ for any $ \alpha\in\Omega^1_{\mathsf{C}^{\infty}_0}(M)$}.
\]
In order to prove $\mathrm{Var}(f)\geq \cdots$, note that $\{\Id f\neq 0\}\subset M$
is an open subset, 
so that there is a sequence $(\psi_n)\subset \mathsf{C}^{\infty}_0(\{\Id f\neq 0\})$
such that 
$0\leq \psi_n\leq 1$ and $\psi_n\to 1$ as $n\to\infty$ pointwise. Then 
\[
\alpha_n:=\f{\psi_n }{|\Id f|}\Id f\in\Omega^1_{\mathsf{C}^{\infty}_0}(\{\Id f\neq 0\})
\subset\Omega^1_{\mathsf{C}^{\infty}_0}(M),
\]
one has $(\Id f,\alpha_n)  \geq 0$, $\left\|\alpha_n\right\|_{\infty}\leq 1$, and we
get
\begin{align}
\int_M |\Id f(x)|_x \mathrm{vol}(\Id x)&
\leq \liminf_{n\to\infty} \int_M (\Id f(x),\alpha_n(x))_x   \mathrm{vol}(\Id x)
\nn\\
&= \liminf_{n\to\infty} \int_M  \overline{ f(x)}\Id^{\dagger} \alpha_n(x) 
\mathrm{vol}(\Id x)
\leq \mathrm{Var}(f), \nn
\end{align}
where we have used Fatou\rq{}s lemma.
\end{proof}

We continue with the distributional properties of the differential of bounded variation
functions: 
In the $m$-dimensional Euclidean situation, it is known (see for example Corollary
\ref{clas-m} below) 
that there is a bijection between functions with bounded variation and the Banach
space of classical 
$\IC^m$-valued Borel measures. On a nonparallelizable manifold this need not be the
case anymore. 
However, we found a \emph{global} statement, which is formulated as Theorem
\ref{Bvmeasures} below, 
and which locally produces precisely the above statements (see also Proposition
\ref{rm}). The essential 
idea is to use a Banach space of \lq\lq{}generalized vector measures on $M$\rq\rq{}.\\
To this end, let $\mathsf{C}_{\infty}$ stand for the class of sections in Hermitian 
vector bundles that vanish at infinity. Clearly, $\Omega^1_{\mathsf{C}_{\infty}}(M)$
becomes a 
complex Banach space with respect to $\left\|\bullet\right\|_{\infty}$. 

\begin{Definition} \label{vm}
The Banach dual $(\Omega^1_{\mathsf{C}_{\infty}}(M))^*$ is called the \emph{space of
generalized vector 
measures on $M$.}
\end{Definition}
We denote the canonical norm on the space of generalized vector measures with 
$\left\|\bullet\right\|_{\infty,*}$.\\
Using Friedrichs-mollifiers and a standard partition of unity argument one finds
that the elements 
of $\mathsf{C}_0(M)$ can be approximated in $\left\|\bullet\right\|_{\infty}$ by
$\mathsf{C}^{\infty}_0(M)$. 
Using now that $\mathsf{C}_0(M)$ is dense $\mathsf{C}_{\infty}(M)$, together with a
second localization 
argument, one gets that $\Omega^1_{\mathsf{C}^{\infty}_0}(M)$ is dense in
$\Omega^1_{\mathsf{C}_{\infty}}(M)$. 
Thus whenever a linear functional $T$ on $\Omega^1_{\mathsf{C}^{\infty}_0}(M)$
satisfies
\[
\left\|T\right\|_{\infty,*}:=\left. \sup\Big\{  
\left|T(\alpha)\right| \right|\alpha\in \Omega^1_{\mathsf{C}^\infty_0}(M),
\|\alpha\|_\infty \leq 1\Big\}<\infty,
\]
it can be uniquely extended to an element of $(\Omega^1_{\mathsf{C}_{\infty}}(M))^*$
with the same norm.\\
Let us furthermore denote with $\IMM(M)$ the space of equivalence classes
$[(\mu,\sigma)]$ of pairs 
$(\mu,\sigma)$ with $\mu$ a finite positive Borel measure on $M$ and $\sigma$ a
Borel section in 
$\mathrm{T}^* M$ with $|\sigma|=1$ $\mu$-a.e. in $M$, where
\begin{align}
(\mu,\sigma)\sim (\mu\rq{},\sigma\rq{}): \ \Leftrightarrow \ &\mu=\mu\rq{}\>\text{
as Borel measures}
\nn\\
& \text{ and $\sigma(x)=\sigma\rq{}(x)$ for $\mu /\mu\rq{}$ a.e. $x\in M$}.\nn
\end{align}

Now we can formulate the following structure theorem, which in particular gives a
global justification of 
Definition \ref{vm}:

\begin{Theorem} \label{Bvmeasures}
\emph{a)} The map
\begin{align}
&\Psi: \IMM(M)\longrightarrow
(\Omega^1_{\mathsf{C}_{\infty}}(M))^*,\>\>\>\Psi[(\mu,\sigma)](\alpha):= 
\int_M (\sigma,\alpha) \Id  \mu \nn
\end{align}
is a well-defined bijection with
$\left\|\Psi[(\mu,\sigma)]\right\|_{\infty,*}=\mu(M)$.\vspace{1.2mm}


\emph{b)} A function $f\in\mathsf{L}^1_{\mathrm{loc}}(M)$ has bounded variation if
and only if 
$\left\|\Id f\right\|_{\infty,*}<\infty$, and then one has 
$\mathrm{Var}(f)=\left\|\Id f\right\|_{\infty,*}$. 
\end{Theorem}
\begin{proof} a) Clearly, $\Psi$ is a well-defined map. We divide the proof into
three parts:\vspace{1mm}

1.\emph{$\Psi$ is surjective}: Let $T$ be a generalized vector measure and consider
the functional given by
\begin{equation}\label{mapsto}
\tilde{\tilde{T}}(f): =\left. \sup\Big\{  \left|T(f\alpha)\right| 
\right|\alpha\in \Omega^1_{\mathsf{C}_\infty}(M),
\|\alpha\|_\infty \leq 1\Big\}, \>\>0\leq f\in \mathsf{C}_{\infty}(M).
\end{equation}
For every test form $\alpha\in \Omega^1_{\mathsf{C}_{\infty}}(M)$ and 
every $0\leq f\in \mathsf{C}_{\infty}(M)$ with $T(f\alpha)\ne 0$ set 
\[
z = \frac{\overline{T(f\alpha)}}{|T(f\alpha)|}\in \IC,
\]
whence 
\[
|T(f\alpha)|=z T(f\alpha) = 
T(zf \alpha)=\mathrm{Re}\Bigl(T(f z \alpha)\Bigr). 
\]
Since $|z|=1$, $z\alpha$ is again an admissible test 
form as well, and we get
\[
|T(f\alpha)| \leq \sup\left\{ \mathrm{Re}\Bigl(T(f\omega)\Bigr)  \Big|
\> \omega\in \Omega^1_{\mathsf{C}_\infty}(M),\ \|\omega\|_\infty \leq 1\right\},
\] 
so that
\[
\tilde{\tilde{T}}(f) = \sup\left\{ \left. \mathrm{Re}\Bigl(T(f\alpha)\Bigr) \right| 
\alpha\in \Omega^1_{\mathsf{C}_\infty}(M),\ \|\alpha\|_\infty \leq 1\right\}.
\]
Thus from demposing the real- and imaginary parts of functions into their positive
and negative parts, 
it follows that $\tilde{\tilde{T}}$ has a unique extension to a positive bounded
linear functional 
$\tilde{T}$ on $\mathsf{C}_\infty(X)$, and by Riesz-Markoff\rq{}s theorem (see
Proposition \ref{rm} a)), 
there is a unique finite positive Borel measure $\mu$ on $M$ such that $\int_M f \Id
\mu = \tilde{T}(f)$ 
for all $f\in \mathsf{C}_{\infty}(M)$. Furthermore, $\mu$ satisfies
\[
\mu(M)=\sup \left\{\left.|\tilde{T}(f)|\ \right|\ f\in\mathsf{C}_{\infty}(M), 
\|f\|_\infty \leq 1 \right\}<\infty.
\]
By the paracompactness of $M$ there is a finite open cover $\bigcup^d_{l=1} U_l=M$
(with $U_l$ possibly 
disconnected and with noncompact closure) such that for any $l$ there is an
orthonormal basis 
$e^{(l)}_1,\dots,  e^{(l)}_m\in\Omega^1_{\mathsf{C}^{\infty}}(U_l)$. We also take a
partition of 
unity $(\psi_l)$ subordinate to $(U_l)$, that is, $\psi_l\in\mathsf{C}^{\infty}(M)$,
$0\leq \psi_l\leq 1$ 
and $\sum_l\psi_l=1$ pointwise, and $\mathrm{supp}(\psi_l)\subset U_l$. Then by the
above considerations, 
the assignment $f\mapsto T(f\psi_l e^{(l)}_j)$, $f\in \mathsf{C}_{\infty}(U_l)$,
extends to a bounded 
linear functional on $\mathsf{L}^1(U_l,\mu)$ with norm $\leq \mu(M)$, so that there
is a 
$\sigma^{(l)}_{j}\in \mathsf{L}^{\infty}(U_l,\mu)$ with $|\sigma^{(l)}_{j}|\leq 1$
$\mu$-a.e. such that
\[
T(f e^{(l)}_j)=\int_{U_l} \sigma^{(l)}_{j}(x)f(x) \mu(\Id x)\>
\text{ for any $f\in \mathsf{C}_{\infty}(U_l)$}.
\]
Now it is easily checked that the Borel $1$-form
$\sigma:=\sum^d_{l=1}\psi_l\sum^m_{j=1}  
\sigma^{(l)}_{j} e^{(l)}_j$ on $M$ satisfies $|\sigma|\leq 1$ $\mu$-a.e. and 
$T=\int_M (\sigma,\bullet)\Id \mu$. Finally, since 
\begin{align}
&\int_M|\sigma(x)|_x\mu(\Id x) \nn\\
&\geq \sup\left. \left\{\left| \int_M (\sigma,\alpha)\Id\mu\right| \>\right|\ 
\alpha\in \Omega^1_{\mathsf{C}_\infty}(M),\ 
\|\alpha\|_\infty \leq 1\right\} \nn\\
&=\sup \left\{\left.|\tilde{T}(f)|\ \right|\ f\in\mathsf{C}_{\infty}(M),
\|f\|_\infty \leq 1 \right\}\nn\\
&= \mu(M),\nn
\end{align}
we get $|\sigma|=1$ $\mu$-a.e., which completes the proof of the asserted
surjectivity of $\Psi$. \vspace{1mm}

2.\emph{$\Psi$ is injective}: If
$\Psi[(\mu,\sigma)](\alpha)=\Psi[(\mu\rq{},\sigma\rq{})](\alpha)$ for all 
$\alpha \in \Omega^1_{\mathsf{C}_{\infty}}(M)$, then one has $\mu=\mu'$ as Borel
measures: Indeed, using the 
above notation, for every Borel set $N\subset M$ we have 
\begin{align}
&\mu(N) \nn\\
&= \sup \left\{\left.\widetilde{\widetilde{\Psi[(\mu,\sigma)]}}(f)\ \right|\
f\in\mathsf{C}_{\infty}(M), 0\leq f\leq 1_N \right\}
\nn\\
&= \sup\left. \Big\{  \left| \Psi[(\mu,\sigma)](f\alpha)\right|\ \right|\
f\in\mathsf{C}_{\infty}(M),\alpha\in\Omega^1_{\mathsf{C}_{\infty}}(M),|\alpha|\leq
1, 0\leq f\leq1_N\Big\}
\nn\\
&=\sup \left.\Big\{  \left|\Psi[(\mu,\sigma)](\beta)\right|\ \right|\
\beta\in\Omega^1_{\mathsf{C}_{\infty}}(M), |\beta|\leq 1_N\Big\}
\nn\\
&=\sup\left\{ \left.\left| \int_M (\sigma',\beta) \Id \mu'\right|\ \right|\
\beta\in\Omega^1_{\mathsf{C}_{\infty}}(M),  |\beta|\leq 1_N\right\}
\nn\\
& \leq \mu'(N).\label{eq}
\end{align}
Exchanging $\mu$ with $\mu'$ we get $\mu(N)=\mu'(N)$ for all Borel sets $N\subset
M$, as claimed. 
To see that $\sigma=\sigma'$ $\mu/\mu'$-a.e., let us observe that
$T:=\Psi[(\mu,\sigma)]$ extends to 
$\Omega^1_{{\mathsf B}_b}(M)$, where ${\mathsf B}_b$ denotes the space of bounded
Borel functions, as  
for every $\alpha \in \Omega^1_{{\mathsf B}_b}(M)$ the function
$x\mapsto(\alpha(x),\sigma(x))_x$ on $M$ belongs 
to $\mathsf{L}^1(M,\mu)$ and we may define $T(\alpha)=\int_M(\alpha,\sigma)\Id\mu$. 
Therefore, we may apply $T$ to $\sigma$ and use that
$T(\sigma)=\Psi[(\mu',\sigma')](\sigma)$ 
and the equality $\mu=\mu'$ to get  
\[
\mu(M) = T(\sigma) = \int_M(\sigma,\sigma')\Id\mu' = \int_M(\sigma,\sigma')\Id\mu 
\leq \int_M|(\sigma,\sigma')|\Id\mu 
\]
and the equality $|(\sigma,\sigma')|=1$ $\mu$-a.e follows. Since $|\sigma|\leq 1$
and $|\sigma'|\leq 1$ we 
deduce $\sigma=\sigma'\ \mu$-a.e.

\vspace{1mm}

3.\emph{ One has $\left\|\Psi[(\mu,\sigma)]\right\|_{\infty,*}=\mu(M)$}: Indeed,
this follows from the bijectivity of $\Psi$ and the proof of the surjectivity.

b) If $\left\|\Id f\right\|_{\infty,*}<\infty$, then clearly one has 
$\mathrm{Var}(f)\leq \left\|\tilde{\Id f}\right\|_{\infty,*}$ by the very definition
of $\Id f$, namely, 
\[
\Id f(\alpha) = \int_M \overline{f(x)}\Id^\dagger \alpha(x)\mathrm{vol}(\Id x)\>
\text{ for any $\alpha\in\Omega^1_{\mathsf{C}^{\infty}_0}(M)$.}
\]
Conversely, if $\mathrm{Var}(f)$ is finite, then by a homogeneity argument, one has
the following estimate 
\[
\left|\int_M \overline{f(x)}\Id^\dagger \alpha(x)\mathrm{vol}(\Id x)\right| 
\leq \mathrm{Var}(f)\|\alpha\|_\infty\>\text{ for any
$\alpha\in\Omega^1_{\mathsf{C}^{\infty}_0}(M)$,}
\]
which implies $\left\|\Id f\right\|_{\infty,*}\leq \mathrm{Var}(f) <\infty$.\end{proof}

If $f$ has bounded variation we will write $|\mathrm{D}f|$ for the measure, and $\sigma_f$ for the $|\mathrm{D}f|$-equivalence class of sections corresponding to $\Psi^{-1}(\Id
f)$, so that we have 
\begin{equation}\label{sigma_fdf}
\Id f(\alpha) = \int_M  (\sigma_f(x), \alpha(x))_x |\mathrm{D}f|(\Id x)\>\>
\text{ for any $ \alpha\in\Omega^1_{\mathsf{C}^{\infty}_0}(M)$}.
\end{equation}
We directly recover the following complex variant of a classical result, which in
particular states that 
locally, complex-valued bounded variation functions can be considered as
$\IR^2$-valued bounded variation 
functions and vice versa:

\begin{Corollary}\label{clas-m} 
Let $M=U$, with $U\subset\IR^m$ a domain with its Euclidean metric, and let
$f\in\mathsf{L}^1_{\mathrm{loc}}(U)$. 
Then one has
\begin{align}
&\mathrm{Var}(f) = \nn\\ 
&\sup\left\{\left. \int_{U}\Big( \mathrm{Re}(f)\mathrm{div}\alpha_1+
\mathrm{Im}(f)\mathrm{div}\alpha_2\Big)\mid_x\Id x
\>\right|    \alpha\in[\mathsf{C}^{\infty}_{0,\IR}(U)]^{2m}, \|\alpha\|_{\infty}\leq
1\right\},\label{char1}
\end{align}
where in the above set any $\alpha\in[\mathsf{C}^{\infty}_{0,\IR}(U)]^{2m}$ is
written as 
$\alpha=(\alpha_1,\alpha_2)$ with $\alpha_j\in[\mathsf{C}^{\infty}_{0,\IR}(U)]^{m}$.
Furthermore, $f$ 
has bounded variation if and only if there is a (necessarily unique) $\IC^m$-valued
Borel measure 
$\mathrm{D}f$ on $U$ such that\footnote{Here, $|\mathrm{D}f|$ stands for the total variation measure corresponding to the vector measure $\mathrm{D} f$ (cf. Section \ref{complexmeas}). Of course this notation is consistent with (\ref{sigma_fdf}).} $\mathrm{grad} f= \mathrm{D} f$ as distributions, and
then it holds 
that $\mathrm{Var}(f) = |\mathrm{D}f|(U)$.
\end{Corollary} 

\begin{proof} If $f$ has bounded variation, then with standard identifications,
Theorem \ref{Bvmeasures} 
implies the existence of $\mathrm{D}f$ in a way that $\mathrm{Var}(f) =
|\mathrm{D}f|(U)$, and Proposition 
\ref{rm} b) implies $|\mathrm{D}f|(U)= |(\mathrm{D}f)_{\IR^{2m}}|(U)$, where
$|(\mathrm{D}f)_{\IR^{2m}}|(U)$ 
is well-known to be equal to the supremum in (\ref{char1}) (see for example the
proof of Proposition 3.6 in 
\cite{ambro}). If $f$ has infinite variation, then we can conclude analogously.
\end{proof}

\begin{Remark}\label{clas-1} 
Note that in the situation of Corollary \ref{clas-m}, the equality of
$\mathrm{Var}(f)$ to the supremum 
in (\ref{char1}) is not immediate from Definition \ref{vari}, where a complex
absolute value appears. But 
it is precisely this equality that makes our definition of variation the natural one
in the case of 
complex-valued functions on the most fundamental level, which is the case of
functions of one 
variable: Indeed, let $I\subset \IR$ be an open interval. Then the characterization
of the variation by 
the supremum in (\ref{char1}) combined with 
Theorem 3.27 in \cite{ambro} (this is a highly nontrivial fact) implies that for any 
$f\in \mathsf{L}^1_{\mathrm{loc}}(I)$ one has 
\begin{align}
\mathrm{Var}(f) = \inf_{f(\bullet)\in f}\sup\left\{\left.
\sum^{n-1}_{j=1}|f(x_{j+1})-f(x_{_j})|\>
\right|\text{$n\geq 2$, $x_1<x_2\dots<x_n$}\right\}.\label{asa}
\end{align}
Note here that $f$ is an equivalence class, so the infimum in \eqref{asa} is taken
among all functions 
coinciding with $f$ a.e. in $I$. Indeed, $\sup\{\dots\}$ depends heavily on the
particular representative 
of $f(\bullet):I\to \IC$ of $f$. 
\end{Remark}

We continue with our general observations. Let $p\in [1,\infty)$ and recall that a
countable system of 
seminorms on $\mathsf{L}^p_{\mathrm{loc}}(M)$ is given through $f\mapsto \int_{K_n}
|f|^p\Id \mu$, where 
$(K_n)$ is an exhaustion of $M$ with relatively compact domains and $\mu$ is a
smooth positive Borel measure 
on $M$, that is, the restriction of $\mu$ to an arbitrary chart has a positive
smooth density function with 
respect to the $m$-dimensional Lebesgue measure. Then the corresponding locally
convex topology does not depend 
on the particular choice of $(K_n)$ and $\mu$, and furthermore for any fixed
$\psi\in\mathsf{C}^{\infty}_0(M)$, 
the map 
\[
\mathsf{L}^{1}_{\mathrm{loc}}(M)\longrightarrow [0,\infty), \> 
f\longmapsto \left|\int_M   \overline{f(x)}\psi(x)\mathrm{vol}(\Id x)\right| 
\]
is continuous. This observation directly implies:

\begin{Proposition}\label{low} For any $p\in [1,\infty)$ the maps 
\begin{align}
\mathsf{L}^p_{\mathrm{loc}}(M)\longrightarrow [0,\infty], \>\>f\longmapsto
\mathrm{Var}(f)\nn\\
\mathsf{L}^p(M)\longrightarrow [0,\infty], \>\>f\longmapsto \mathrm{Var}(f)\nn
\end{align}
are lower semicontinuous with respect to the corresponding canonical topologies.
\end{Proposition}

We also have the following fact, which follows easily from the completeness of
$\mathsf{L}^1(M)$ 
and Proposition \ref{low}:

\begin{Proposition}\label{ban} The space
\[
\mathsf{BV}(M):=\left.\Big\{f\right| f\in\mathsf{L}^{1}(M),
\mathrm{Var}(f)<\infty\Big\}
\]
is a complex Banach space with respect to the norm 
$\left\|f\right\|_{\mathsf{BV}}:=\left\|f\right\|_1+\mathrm{Var}(f)$.
\end{Proposition}

Next we shall record that first-order $\mathsf{L}^1$-Sobolev functions belong to
$\mathsf{BV}(M)$ with the same norm. 
To this end, for any $p\in [1,\infty)$ we denote the complex Banach space of first order $\mathsf{L}^p$-Sobolev functions with 
\[
\mathsf{W}^{1,p}(M):=\left.\Big\{f\right| f\in\mathsf{L}^{p}(M), \Id f\in
\Omega^1_{\mathsf{L}^{p}}(M)\Big\}
\]
with its canonical norm 
$\left\|f\right\|_{1,p}:=\left\|f\right\|_p+\left\|\Id f\right\|_p$. Furthermore, we
define 
$\mathsf{H}^{1,p}(M)\subset \mathsf{W}^{1,p}(M)$ as the closure of the space of
functions 
$f\in\mathsf{L}^{p}(M)\cap\mathsf{C}^{\infty}(M)$ such that $\Id f\in
\Omega^1_{\mathsf{L}^{p}}(M)$ 
with respect to $\left\|\bullet\right\|_{1,p}$. The equality
$\mathsf{W}^{1,p}=\mathsf{H}^{1,p}$ is known to hold on open subsets of the Eucliden $\IR^m$ by Meyers-Serrin\rq{}s Theorem \cite{meyers}. It seems to be unknown whether this extends to abitrary $M$. However, one has the following result under geodesic completenes, which should be known, but which we have not been able to find a direct reference for (note here that Theorem 1 in \cite{aubin} only states that $\mathsf{C}^{\infty}_0(M)$ is dense in $\mathsf{H}^{1,p}(M)$). It relies on the existence of first order cut-off functions:

\begin{Proposition}\label{densi} If $M$ is geodesically complete, then $\mathsf{C}^{\infty}_0(M)$ is dense in $\mathsf{W}^{1,p}(M)$, in particular, one has $\mathsf{W}^{1,p}(M)=\mathsf{H}^{1,p}(M)$. 
\end{Proposition}

\begin{proof} Under geodesic completeness, there is a sequence of functions $(\psi_n)\subset
\mathsf{C}^{\infty}_0(M)$ with 
$0\leq \psi_n\leq 1$, $\psi_n\to 1$ 
pointwise and $\left\|\Id\psi_n\right\|_{\infty}\to 0$ as $n \to \infty$ (see
\cite{shub}, Proposition 4.1). For $f\in \mathsf{W}^{1,p}(M)$ let $f_n:=\psi_nf$. Then the Sobolev product rule
$$
\Id f_n= f\Id \psi_n+ \psi_n \Id f
$$
implies that $f_n\in\mathsf{W}^{1,p}_0(M)$ (the compactly supported elements in $\mathsf{W}^{1,p}(M)$!) and also that $f_n\to f$ in $\left\|\bullet\right\|_{1,p}$, the latter from dominated convergence. Thus $\mathsf{W}^{1,p}_0(M)$ is dense in $\mathsf{W}^{1,p}(M)$ and it remains to show that functions from the former space can be approximated by functions in $\mathsf{C}^{\infty}_0(M)$. However, now one can use a partition of unity argument corresponding to a \emph{finite} atlas for $M$ to see that it is sufficient to prove that for an open subset $U$ of the Euclidean $\IR^m$, the space $\mathsf{C}^{\infty}_0(U)$ is dense in the normed space $\mathsf{W}^{1,p}_0(U)$. The latter fact is well-known.
\end{proof}

The following proposition completely clarifies the connection between Sobolev- and $\mathsf{BV}$-functions:

\begin{Proposition}\label{so} \emph{a)} One has $\left\|f\right\|_{\mathsf{BV}}= \left\|f\right\|_{1,1}$ 
for all $f\in\mathsf{W}^{1,1}(M)$. In particular, $\mathsf{H}^{1,1}(M)$ and
$\mathsf{W}^{1,1}(M)$ are closed subspaces of $\mathsf{BV}(M)$. \vspace{1mm}

\emph{b)} Any $f\in \mathsf{BV}(M)$ is in $\mathsf{W}^{1,1}(M)$, if and only if with the notation from \eqref{sigma_fdf} it holds that $|\mathrm{D} f|\ll
\mathrm{vol}$ as Borel measures.
\end{Proposition}

\begin{proof} a) As one has
$$
\Id f(\alpha)=\int_M (\sigma(x),\alpha(x))_x \mu(\Id x)\>
\text{ for any $\alpha\in\Omega^1_{\mathsf{C}^{\infty}_0}(M)$,}
$$
where $\sigma:=\Id f/|\Id f|$ and $\mu:=|\Id f|\mathrm{vol}$, the claim follows immediately from Theorem \ref{Bvmeasures}.

b) In view of  \eqref{sigma_fdf}, if for some $0\leq \rho \in \mathsf{L}^1(M)$ one has $|\mathrm{D} f|=\rho\, \mathrm{vol}$, then $\Id f = \rho\, \sigma_f$ is integrable and $f$ is Sobolev. The other direction follows directly from the proof of part a).
\end{proof}

We close this section with three results on the variation of \emph{globally
integrable} functions 
that all additionally require \emph{geodesic completeness}. Firstly, in the latter
situation, the 
variation can be approximated simultaniously to the $\mathsf{L}^1$-norm by the
corresponding data 
of smooth compactly supported functions:

\begin{Theorem}\label{appros} If $M$ is geodesically complete, then 
for any $f\in \mathsf{L}^1(M)$ there is a sequence $(f_n)\subset
\mathsf{C}^{\infty}_0(M)$ 
such that $f_n\to f$ in $\mathsf{L}^1(M)$ and $\mathrm{Var}(f_n)\to \mathrm{Var}(f)$
as $n\to\infty$. 
\end{Theorem}

\begin{proof} If $\mathrm{Var}(f)=\infty$, then any sequence 
$(f_n)\subset \mathsf{C}^{\infty}_0(M)$ such that $f_n\to f$ in $\mathsf{L}^1(M)$
satisfies 
$\mathrm{Var}(f_n)\to \infty$ in view of Proposition \ref{low}.\\
For the case $\mathrm{Var}(f)<\infty$, let us remark that the statement is proved in 
\cite[Proposition 1.4]{mira} for $\mathsf{BV}_{\IR}(M)$, the real-valued elements of 
$\mathsf{BV}(M)$. However, one can use the same localization argument in our
complex-valued 
situation to reduce the assertion to domains in $\IR^m$ (the geodesic completeness
is used 
precisely in this highly nonstandard localization argument). In the latter case, the
assertion 
follows from combining our Corollary \ref{clas-m} above with suitable known
approximation results, 
which are available for $\IR^2$-valued $\mathsf{BV}$ functions in the Euclidean
setting (cf. 
Theorem 3.9 in \cite{ambro}) or for real-valued $\mathsf{BV}$ functions with respect
to 
\emph{weighted} variation (cf. Theorem 3.4 in \cite{baldi}). Indeed, any of the
latter two 
results can be easily generalized to cover the vector-valued weighted case.
\end{proof}
Note that Theorem \ref{appros} does not imply that $\mathsf{C}^{\infty}_0(M)$ is
dense in $\mathsf{BV}(M)$.\vspace{1.1mm}

We directly get the following corollary from combining the lower semi-continuity of
the variation 
with Proposition \ref{d14}, which seemingly cannot be deduced in an elementary way,
that is, without the above 
approximation result:

\begin{Corollary}\label{eqw} Let $g$ denote the underlying Riemannian structure on
$M$ and let $(M,g)$ be geodesically complete. If $g\rq{}$ is another 
Riemannian structure on $M$ which is quasi-isometric to $g$, that is, if there are
$C_1,C_2>0$ such that for all $x\in M$ one has $C_1g_x\leq g_x\rq{}\leq C_2g_x$ as
norms\footnote{Note here that quasi-isometric Riemannian structures produce
equivalent $\mathsf{L}^p$-norms.}, then the corresponding norms 
$\left\|\bullet\right\|_{\mathsf{BV}}$ and
$\left\|\bullet\right\|\rq{}_{\mathsf{BV}}$ are equivalent.
\end{Corollary}

Finally, we note that geodesic completeness and global integrability admit an
enlargement of the admissible 
class of test-$1$-forms, 
a result that we will also use in the proof of our main result. To this end, let
\[
\Omega^1_{\mathrm{bd}}(M):=\left.\Big\{ \alpha  \right|
\alpha\in\Omega^1_{\mathsf{C}^{\infty}
\cap\mathsf{L}^{\infty}}(M), \Id^{\dagger}\alpha\in \mathsf{L}^{\infty}(M)\Big\}.
\]
Now the following fact can be easily deduced from the existence of 
first order cut-off functions:

\begin{Lemma}\label{d12} 
If $M$ is geodesically complete, then for any $f\in \mathsf{L}^1(M)$ one has 
\[
\mathrm{Var}(f)=\sup\left\{\left.\left|\int_M  \overline{ f(x)}\Id^{\dagger} 
\alpha(x)  \mathrm{vol}(\Id x)\right|\>\right| \alpha\in\Omega^1_{\mathrm{bd}}(M), 
\left\|\alpha\right\|_{\infty}\leq 1  \right\}.\nn
\]
\end{Lemma}

\begin{proof} The proof is essentially the same as the proof of Lemma 3.1 in
\cite{carbo}, which considers 
the real-valued situation. We repeat the simple argument for the convenience of the
reader.\\
 The inequality $\mathrm{Var}(f)\leq \dots$ is trivial. For $\mathrm{Var}(f)\geq
\dots$, we take a sequence 
of first order cut-off functions  $(\psi_n)$ as in the proof of Proposition \ref{densi} and let $
\alpha\in\Omega^1_{\mathrm{bd}}(M)$ be such 
that $\left\|\alpha\right\|_{\infty}\leq 1$. Then one has 
\[
\Id^{\dagger}  (\psi_n\alpha)= \psi_n\Id^{\dagger}  \alpha-\alpha (X^{\Id \psi_n}),
\]
 where $X^{\Id \psi_n}$ is the smooth vector field on $M$ corresponding to $\Id
\psi_n$ via the Riemannian 
structure, and one gets
\[
\left|\int_M  \overline{ f(x)}\Id^{\dagger} \alpha(x)  \mathrm{vol}(\Id x)\right|
=\lim_{n\to \infty}\left|\int_M  \overline{ f(x)}\Id^{\dagger}  (\psi_n\alpha)(x) 
\mathrm{vol}(\Id x)\right|
\leq \mathrm{Var}(f),
\]   
where the equality follows from dominated convergence.
\end{proof}

\section{The heat semigroup characterization of the variation}\label{haup}

We now come to the formulation and the proof of the main result of this note: A heat
semigroup 
characterization of the variation of $\mathsf{L}^1$- functions to a class of
Riemannian manifolds 
with possibly unbounded from below Ricci curvature. To be precise, we will allow
certain negative 
parts of the Ricci cuvature to be in the Kato class of $M$. To this end we first
recall the 
definition of the minimal positive heat kernel $p(t,x,y)$ on $M$: Namely, $p(t,x,y)$
can be defined 
\cite{grig} as the pointwise minimal function 
\[
p(\bullet,\bullet,\bullet):(0,\infty)\times M\times M\longrightarrow  (0,\infty) 
\]
with the property that for all fixed $y\in M$, the function $p(\bullet,\bullet,y)$
is a classic 
(= $\mathsf{C}^{1,2}$) solution of 
\[
\partial_t u(t,x)=\f{1}{2}\Delta u(t,x),\>\>\lim_{t\to 0+} u(t,\bullet)=\delta_y. 
\]
It follows from parabolic regularity that $p(t,x,y)$ is smooth in $(t,x,y)$, and
furthermore 
$p(t,\bullet,\bullet)$ is the unique continuous version of the integral kernel of
$\mathrm{e}^{-t H}$. 
The reader should notice that the strict positivity of 
$p(t,x,y)$ follows from the connectedness of $M$. Now we can define:

\begin{Definition} A Borel 
function $w:M\to \IC$ is said to be in the {\it Kato class} $\mathcal{K}(M)$ of $M$,
if 
\begin{align}
\lim_{t\to 0+}\sup_{x\in M} \int^t_0\int_M p(s,x,y) |w(y)|\mathrm{vol}(\Id y)\Id s =0. 
\end{align}
\end{Definition}

It is easily seen \cite{G1} that one always has the inclusions
\[
\mathsf{L}^{\infty}(M)\subset \mathcal{K}(M) \subset \mathsf{L}^{1}_{\mathrm{loc}}(M),
\]
but in typical applications one can say much more. To make the latter statement
precise, for any 
$p\in [1,\infty)$ let $\mathsf{L}^p_{\mathrm{u,loc}}(M)$ denote the space of uniformly
locally $p$-integrable 
functions on $M$, that is, a Borel function $w:M\to\IC$ is in
$\mathsf{L}^p_{\mathrm{u,loc}}(M)$, if and 
only if    
\begin{align}
\sup_{x\in
M} \int_{\mathrm{K}_1(x)}\left|w(y)\right|^p\mathrm{vol}(\Id
y)<\infty. 
\end{align}
Note the simple inclusions
\[
\mathsf{L}^p(M)\subset \mathsf{L}^p_{\mathrm{u,loc}}(M)\subset
\mathsf{L}^p_{\mathrm{loc}}(M).
\]
Now one has the following result, which essentially states that a Gaussian upper
bound for $p(t,x,y)$ 
implies $\mathsf{L}^{p}(M)\subset \mathcal{K}(M)$ for suitable $p=p(m)$, and that
with a little more 
control on the Riemannian structure one even has
$\mathsf{L}^{p}_{\mathrm{u,loc}}(M)\subset \mathcal{K}(M)$ 
(cf. Proposition 2.4 in \cite{post}):

\begin{Proposition}\label{katt} Let $p$ be such that $p\geq 1$ if $m=1$, and $p>m/2$
if $m\geq 2$.\\
\emph{a)} If there is $C>0$ and a $t_0>0$ such that for all $0<t \leq t_0$ and all
$x\in M$ one has 
$ p(t,x,x)\leq Ct^{ -\f{m}{2} }$, then one has 
\begin{align}
\mathsf{L}^{p}(M)+\mathsf{L}^{\infty}(M)\subset \mathcal{K}(M).\label{incl}
\end{align}
\emph{b)} Let $M$ be geodesically complete, and assume that there are constants
$C_1,\dots ,C_6, t_0 >0$ 
such that for all $0<t \leq t_0$, $x,y\in M$, $r>0$ one has 
$\mathrm{vol}(\mathrm{K}_r(x))\leq C_1 r^m \mathrm{e}^{C_2 r}$ and
\begin{align}
C_3t^{ -\f{m}{2} } \mathrm{e}^{-C_4 \f{ \Id(x,y)^2}{t}}\leq p(t,x,y)
\leq C_5t^{ -\f{m}{2} } \mathrm{e}^{-C_6 \f{ \Id(x,y)^2}{t}}.\nn
\end{align}
Then one has
\begin{align}
\mathsf{L}^{p}_{\mathrm{u,loc}}(M)+\mathsf{L}^{\infty}(M)\subset
\mathcal{K}(M).\label{incl2}
\end{align}
\end{Proposition}

We refer the reader to \cite{G1} and particularly to \cite{kt} for several further
global and local 
aspects on $\mathcal{K}(M)$.\\
In the sequel, we will consider the Ricci curvature $\IRR$ of $M$ as a smooth,
self-adjoint section in 
the smooth complex vector bundle $\mathrm{End}\left( \mathrm{T}^* M\right)\to M$,
whose quadratic form 
is defined pointwise through the trace of the Riemannian curvature tensor of $M$. \\
With these notions at hand, the following dynamical characterization of the
variation of globally integrable 
functions is the main result of this note:

\begin{Theorem}\label{BV} Let $M$ be geodesically complete and assume that $\IRR$
admits a decomposition 
$\IRR=\IRR_1-\IRR_2$ into pointwise self-adjoint Borel sections $\IRR_1,\IRR_2\geq
0$ in 
$\mathrm{End}( \mathrm{T}^* M)$ such that $|\IRR_2|\in\mathcal{K}(M)$. Then for any 
$f\in \mathsf{L}^1(M)$ one has
\begin{align}
\mathrm{Var}(f)=\lim_{t\to 0+} \int_M\left|\Id \mathrm{e}^{-tH}f(x)\right|_x
\mathrm{vol}(\Id x).\label{hap}
\end{align}
\end{Theorem}

\begin{Remark}\label{bem} 
If $M$ is the Euclidean $\IR^m$, then (\ref{hap}) has been proven 
(for real-valued $f\rq{}s$) by De Giorgi \cite{giorgi} in 1954. De Giorgi\rq{}s
result has 
been extended to geodesically complete Riemannian manifolds first by 
Miranda/ the second author/Paronetto/Preunkert \cite{mira} in 2007, under the
assumptions that $M$ has 
Ricci curvature $\IRR$ bounded below and satisfies the nontrapping condition
\begin{align}
\inf_{x\in M} \mathrm{vol} (\mathrm{K}_1(x))>0.\label{trap} 
\end{align}
Again in 2007, Carbonaro-Mauceri \cite{carbo} have removed condition \eqref{trap},
giving a  
much simpler proof that relies on an $\mathsf{L}^\infty$-estimate for
$\mathrm{e}^{-tH_1}$ 
due to Bakry \cite{Bakry}. \\
The point we want to make here is that a large part of the technique from
\cite{carbo} is flexible enough 
to deal with certain unbounded \lq\lq{}negative parts\rq\rq{} of $\IRR$. The
essential observation is that 
in view of Weitzenb\"ock\rq{}s formula for $-\Delta_1$, $\mathrm{e}^{-t H_1}$
becomes a generalized 
Schr\"odinger semigroup with potential $\IRR/2$, which, under our assumptions on
$M$, is given by a 
Feynman-Kac type path integral formula. This follows from the abstract work of the
first author on generalized 
Schr\"odinger semigroups \cite{G2}. Through semigroup domination, the latter formula
makes it possible 
to prove (see Lemma \ref{smoo} below) a bound of the form
\[
\left\|\mathrm{e}^{-t H_1}\mid_{\Omega^1_{\mathsf{L}^{2}\cap
\mathsf{L}^{\infty}}(M)}\right\|_{\infty}
\leq \delta \mathrm{e}^{tC(\delta)}\>\text{ for all $t\geq 0$, $\delta>1$,}
\]
which is weaker for small times than the above mentioned
$\mathsf{L}^\infty$-estimate by Bakry for the case $\IRR\geq -C$ (the latter is the
form $\dots\leq \mathrm{e}^{tC}$), but 
turns out to be just enough to extend a large part of the methods of \cite{carbo} to
our more general setting. 
Finally, let us also point out that heat semigroup characterizations of
$\mathsf{BV}$ have also been 
derived in other situations than functions on Riemannian manifolds (cf. \cite{FH}
\cite{BVWiener1} 
\cite{Carnot}). We propose two further extensions of the setting of Theorem \ref{BV}:
\begin{itemize} 
\item Definition \ref{vari} suggests that one can define the notion of  a
\lq\lq{}$D$-variation\rq\rq{} 
for sections in vector bundles, instead of functions, where \lq\lq{}$\Id$\rq\rq{}
has to be replaced by 
an appropriate first order linear differential operator $D$ acting between sections.
Here, in principle, 
the results from \cite{thal} on path integral formulae for the derivatives of
geometric Schr\"odinger 
semigroups could be very useful. 

\item As all of the data in (\ref{hap}) have analogues (see for example
\cite{dod,ognjen}) on discrete metric graphs, it 
would certainly be also interesting to see to what extent such a result can be
proved in the infinite 
discrete setting. 
\end{itemize}
\end{Remark}

Before we come to the proof of Theorem \ref{BV}, we continue with several
consequences of the latter result. Firstly, in view of Proposition \ref{katt}, we
directly get the following criterion:

\begin{Corollary}\label{cor1} \emph{a)} Under the assumptions of Proposition
\ref{katt} a), assume that $\IRR$ admits a decomposition 
$\IRR=\IRR_1-\IRR_2$ into pointwise self-adjoint Borel sections $\IRR_1,\IRR_2\geq
0$ in 
$\mathrm{End}( \mathrm{T}^* M)$ such that
$|\IRR_2|\in\mathsf{L}^p(M)+\mathsf{L}^{\infty}(M)$. Then one has (\ref{hap}).\\
\emph{b)} Under the assumptions of Proposition \ref{katt} b), assume that $\IRR$
admits a decomposition 
$\IRR=\IRR_1-\IRR_2$ into pointwise self-adjoint Borel sections $\IRR_1,\IRR_2\geq
0$ in 
$\mathrm{End}( \mathrm{T}^* M)$ such that
$|\IRR_2|\in\mathsf{L}^{p}_{\mathrm{u,loc}}(M)+\mathsf{L}^{\infty}(M)$. Then one has
(\ref{hap}).
\end{Corollary}

We remark that a somewhat comparable assumption on the Ricci curvature as in part a)
of the above corollary 
has also been made in Theorem 2.1 in \cite{coul}, where the authors prove certain
large time bounds on the 
integral kernel of $\mathrm{e}^{-t H_1}$.\vspace{1.2mm}

Theorem \ref{BV} makes it also possible to derive a stability of (\ref{hap}) under
certain conformal 
transformations which is particularly useful in the Euclidean $\IR^m$. 

\begin{Remark}1. If $g$ denotes the fixed Riemannian structure on $M$ and if 
$\psi\in\mathsf{C}^{\infty}_{\IR}(M)$, then we can define a new Riemannian structure
on $M$ by setting 
$g_{\psi}:=\mathrm{e}^{2\psi} g$. Cearly, a section in $\mathrm{End}(\mathrm{T}^*M)$
is self-adjoint with 
respect to $g$, if and only if it is self-adjoint with respect to $g_{\psi}$.

2. The Riemannian structures $g$ and $g_{\psi}$ are quasi-isometric if $\psi$ is
bounded, so that then 
$\mathsf{L}^p(M;g_{\psi})=\mathsf{L}^p(M)$, as well as 
$\mathsf{L}^p_{\mathrm{u,loc}}(M;g_{\psi})=\mathsf{L}^p_{\mathrm{u,loc}}(M)$ for all
$p$. Clearly, the 
boundedness of $\psi$ also implies that a self-adjoint section in
$\mathrm{End}(\mathrm{T}^*M)$ is bounded 
from below with respect to $g$, if and only if it is so with respect to $g_{\psi}$.

3. If we denote with $\IRR_{\psi}$ the Ricci curvature with respect to $g_{\psi}$,
then one has the 
perturbation formula (see for example Theorem 1.159 in \cite{besse})
$\IRR_{\psi}=\IRR +\ITT_{\psi}$, 
where $\ITT_{\psi}$ is the smooth self-adjoint section in $\mathrm{End}(\mathrm{T}^*
M)$ given by
\begin{align}
\ITT_{\psi}:=(2-m)\Big(\mathrm{Hess}(\psi)-\Id \psi\otimes \Id \psi\Big)-\Big(\Delta
\psi+
(m-2)|\Id \psi|^2\Big) g.\label{decom}
\end{align}
It is clear from (\ref{decom}) that $\IRR_{\psi}$ need not be bounded from below,
even if $\IRR$ is bounded 
from below and $\psi$ is bounded.
\end{Remark}

Now we can prove the following result which uses the machinery of parabolic Harnack
inequalities:

\begin{Corollary}\label{cor2} Let $\psi\in\mathsf{C}^{\infty}_{\IR}(M)$ be bounded,
let $M$ be geodesically 
complete and let $p$ be such that $p\geq 1$ if $m=1$, and $p>m/2$ if $m\geq 2$.
Furthermore, assume that 
there are $C_1,C_2,R>0$ with the following property: one has $\IRR\geq -C_1$ and 
\begin{align}
\mathrm{vol}(\mathrm{K}_r(x))\geq C_2 r^{m}\>\>\text{ for all $0<r\leq R$, $x\in
M$.}\label{m0}
\end{align}
If $\ITT_{\psi}$ admits a decomposition 
$\ITT_{\psi}=\ITT_{1}-\ITT_{2}$ into pointwise self-adjoint Borel sections
$\ITT_{1},\ITT_{2}\geq 0$ in 
$\mathrm{End}( \mathrm{T}^* M)$ such that 

$$|\ITT_{2}|\in\mathsf{L}^{p}_{\mathrm{u,loc}}(M;g_{\psi})+\mathsf{L}^{\infty}(M;g_{\psi}),$$
then for any 
$f\in \mathsf{L}^1(M;g_{\psi})$ one has (\ref{hap}) with respect to $g_{\psi}$.
\end{Corollary}

\begin{proof} Firstly, we note the classical fact that $\IRR\geq -C_1$ implies
Li-Yau\rq{}s estimate for 
all $t>0, x,y\in M$,
\begin{align}
\f{C_3}{\mathrm{vol}(\mathrm{K}_{\sqrt{t}}(x))}\mathrm{e}^{-C_4 \f{
\Id(x,y)^2}{t}}\leq p(t,x,y)
\leq \f{C_5}{\mathrm{vol}(\mathrm{K}_{\sqrt{t}}(x))}\mathrm{e}^{-C_6 \f{
\Id(x,y)^2}{t}},\label{m1}
\end{align}
thus using 
\begin{align}
\mathrm{vol}(\mathrm{K}_r(x))\leq C_7 r^m \mathrm{e}^{C_8 r}\text{ for all
$r>0$},\label{m2}
\end{align}
which is a simple consequence of Bishops\rq{}s volume comparison theorem (see for
example p. 7 in 
\cite{hebey}), we can deduce the inequality
\begin{align}
C_9 t^{ -\f{m}{2} } \mathrm{e}^{-C_{10} \f{ \Id(x,y)^2}{t}}\leq p(t,x,y)
\leq  C_{11} t^{ -\f{m}{2} } \mathrm{e}^{- C_{12} \f{ \Id(x,y)^2}{t}}\>\>
\text{ for all $0< t\leq 1$.}\label{m3}
\end{align}
By Theorem 5.5.3 in \cite{sal}, Li-Yau\rq{}s inequality is equivalent to the
conjunction of the local 
Poincar\'e inequality and volume doubling, which by Theorem 5.5.1 in \cite{sal} is
equivalent to the 
validity of the parabolic Harnack inequality. The latter inequality is stable under
a change to an 
quasi-isometric Riemannian structure by Theorem 5.5.9 in \cite{sal}, so that we also
have (\ref{m1}) with 
respect to $g_{\psi}$. Again using the quasi-isometry of the Riemannian structures,
it is clear that 
we also have (\ref{m0}) and (\ref{m2}), and thus (\ref{m3}) with respect to
$g_{\psi}$. But now we 
can use Corollary \ref{cor1} to deduce (\ref{hap}) with respect to $g_{\psi}$,
keeping in mind that 
the negative part $\IRR_-$ is bounded by assumption.
\end{proof}

The rest of this paper is devoted to the proof of Theorem \ref{BV}, which will
require two more 
auxiliary results. To this end, we have to introduce some probabilistic notation
first.\\
Let $(\Omega,\IFF,\IFF_*,\mathbb{P})$ be a filtered probability space which
satisfies the usual assumptions. 
We assume that $(\Omega,\IFF,\IFF_*,\mathbb{P})$ is chosen in a way such that it
carries an appropriate 
family of Brownian motions $B(x)\colon [0,\zeta(x))\times \Omega \to M$, $x\in M$,
where 
$\zeta(x):\Omega\to [0,\infty]$ is the lifetime of $B(x)$. The well-known relation
\cite{Hsu} 
\begin{align}
  \mathbb{P}\{B_t(x)\in N, t<\zeta(x)\}=\int_N p(t,x,y)\mathrm{vol}(\Id y)\>\>\text{
for any Borel set $N\subset M$}\nn
\end{align}
implies directly that for a Borel function 
$w:M\to \IC$ one has $w\in\mathcal{K}(M)$, if and only if
\[
\lim_{t\to 0+}\sup_{x\in M} \mathbb{E}\left[\int^t_0
\left|w(B_s(x))\right|1_{\{s<\zeta(x)\}}\Id s\right] =0,
\]
which is the direct link between Theorem \ref{BV} and probability theory. We will
need the following 
subtle generalization of Proposition 2.5 from \cite{G2}, which does not require any
control on the 
Riemannian structure of $M$:

\begin{Lemma}\label{xdd} 
For any $v\in\mathcal{K}(M)$ and any $\delta>1$ there is a $C(v,\delta)>0$ such that
for all $t\geq 0$,
\begin{align}
\sup_{x\in M} \mathbb{E}\left[\mathrm{e}^{\int^t_0 \left|v(B_s(x))\right|\Id
s}1_{\{t<\zeta(x)\}}\right]
\leq \delta \mathrm{e}^{tC(v,\delta)}. \label{way0}
\end{align}
\end{Lemma}

\begin{proof} The proof is an adaption of that of Proposition 2.5 from \cite{G2}
(see particularly 
also \cite{dem}, \cite{sznit}, \cite{ai}). We give a detailed proof here for the
convenience of the reader. 
Let us first state two abstract facts:\\
1. With $\hat{M}=M\cup\{\infty_M\}$ the Alexandroff compactification of $M$, we can
canonically extend 
any Borel function $w:M\to \IC$ to a Borel function $\hat{w}:M\to \IC$ by setting 
$\hat{w}(\infty_M)=0$, and $B(x)$ to a process $\hat{B}(x):[0,\infty)\times
\Omega\to \hat{M}$ 
by setting $\hat{B}_s(x)(\omega):=\infty_M$, if $s\geq \zeta(x)(\omega)$. Then one
trivially has 
\begin{align}
\mathbb{E}\left[\mathrm{e}^{\int^t_0 \left|w(B_s(x))\right|\Id
s}1_{\{t<\zeta(x)\}}\right]
\leq  \mathbb{E}\left[\mathrm{e}^{\int^t_0 \left|\hat {w}(\hat{B}_s(x))\right|\Id
s}\right].\label{way}
\end{align}
2. For any Borel function $w:M\to \IC$ and any $s\geq 0$ let
\begin{align}
D(w,s):=&\sup_{x\in M}\mathbb{E}\left[\int^s_0\left|\hat{w}(\hat{B}_r(x))\right|\Id
r\right]\nn\\
&=\sup_{x\in M}\mathbb{E}\left[\int^t_0 \left|w(B_s(x))\right|1_{\{s<\zeta(x)\}}\Id
s\right]\in [0,\infty],\nn
\end{align}
and
\begin{align}
\tilde{D}(w,s):=
\sup_{x\in
M}\mathbb{E}\left[\mathrm{e}^{\int^s_0\left|\hat{w}(\hat{B}_r(x))\right|\Id
r}\right]
\in [0,\infty].\nn
\end{align}
Then \emph{Kas\rq{}minskii\rq{}s Lemma} states that the following assertion holds: 
\begin{align}
\text{For any $s>0$ with $D(w,s)<1$ one has $\tilde{D}(w,s)\leq \f{1}{1-D(w,s).}$ 
}\label{way2}
\end{align}
This estimate can be proved as follows: For any $n\in\IN$ let
\[
s\sigma_n:=\Big\{q=(q_1,\dots,q_n)\left|\> 0\leq q_1\leq\dots\leq q_n\leq
s\Big\}\right.\subset \IR^n
\]
denote the $s$-scaled standard simplex. Then it is sufficient to prove that for all
$n$ one has 
\begin{align}
\tilde{D}_n(w,s)&:=
\sup_{x\in
M}\int_{s\sigma_n}\mathbb{E}\left[\left|\hat{w}(\hat{B}_{q_1}(x))\right|\dots
\left|\hat{w}
(\hat{B}_{q_n}(x))\right|\right]\Id^n q\nn\\
&\leq \f{1}{1-D(w,s)} \tilde{D}_{n-1}(w,s).
\end{align}
But the Markoff property of $B(x)$ implies
\begin{align}
\tilde{D}_n(w,s)&=\sup_{x\in M}\int_{s\sigma_{n-1}}\mathbb{E}
\left[\left|\hat{w}(\hat{B}_{q_1}(x))\right|\dots
\left|\hat{w}(\hat{B}_{q_{n-1}}(x))\right|\times\right.\nn\\
&\hspace{28mm}\left.\times
\mathbb{E}\left[\int^{s-q_{n-1}}_0\left|\hat{w}(\hat{B}_u(y))\right|
\Id u\right]\mid_{y=\hat{B}_{q_{n-1}}(x)}\right]\Id^{n-1}q\nn\\
&\leq \f{1}{1-D(w,s)} \tilde{D}_{n-1}(w,s),
\end{align}
\vspace{1.2mm}
which proves Kas\rq{}minskii\rq{}s lemma.\\
Using the two observations above, the actual proof of (\ref{way0}) can be carried
out as follows: 
By the Kato property of $v$, we can pick an $s(v,\delta)>0$ with
$D(v,s(v,\delta))<1-1/\delta$. 
Let $n\in\IN$ be large enough with $t\leq (n+1) s(v,\delta)$. Then the Markoff
property of $B(x)$ 
and Kas\rq{}minskii\rq{}s Lemma imply 
\begin{align}
&\tilde{D}(v,t) \nn\\
&\leq \tilde{D}(v,(n+1) s(v,\delta))\nn\\
&=\sup_{x\in M}
\mathbb{E}\left[\mathrm{e}^{\int^{ns(v,\delta)}_0\left|\hat{v}(\hat{B}_r(x))\right|
\Id
r}\mathbb{E}\left[\mathrm{e}^{\int^{s(v,\delta)}_0\left|\hat{v}(\hat{B}_r(y))\right|\Id
r}\right]
\mid_{y=\hat{B}_{ns(v,\delta)}(x)}\right]\nn\\
&\leq \f{1}{1-D(v,s(v,\delta))} \tilde{D}(v,n s(v,\delta))  \nn\\
&=\f{1}{1-D(v,s(v,\delta))}\times\nn\\
&\>\>\>\>\>\times \sup_{x\in M}
\mathbb{E}\left[\mathrm{e}^{\int^{(n-1)s(v,\delta)}_0\left|\hat{v}
(\hat{B}_r(x))\right|\Id
r}\mathbb{E}\left[\mathrm{e}^{\int^{s(v,\delta)}_0\left|\hat{v}(\hat{B}_r(y))
\right|\Id r}\right]\mid_{y=\hat{B}_{(n-1)s(v,\delta)}(x)}\right]\nn\\
&\leq \dots\>\text{($n$-times)}\nn\\
&\leq \f{1}{1-D(v,s(v,\delta))}\left(\f{1}{1-D(v,s(v,\delta))}\right)^{n}  \nn\\
&\leq \f{1}{1-D(v,s(v,\delta))} \mathrm{e}^{ \f{t}{s(v,\delta)} 
\mathrm{log}\left(\f{1}{1-D(v,s(v,\delta))}\right) }\nn\\
& < \delta \ \mathrm{e}^{ \f{t}{s(v,\delta)} \mathrm{log}\left( 
\f{1}{1-D(v,s(v,\delta))}  \right) },\nn
\end{align}
which proves (\ref{way0}) in view of (\ref{way}).
\end{proof}

The latter result will be used to deduce:

\begin{Lemma}\label{smoo} 
Under the assumptions of Theorem \ref{BV}, for any $\delta>1$ there is a
$C(\delta)>0$ such that 
for any $t\geq 0$ and any $\alpha\in \Omega^1_{\mathsf{L}^{2}\cap
\mathsf{L}^{\infty}}(M)$ one has 
\begin{align}
\left\|\mathrm{e}^{-t H_1}\alpha\right\|_{\infty}
\leq \delta \mathrm{e}^{tC(\delta)} \left\|\alpha\right\|_{\infty}.
\end{align}
\end{Lemma}

\begin{proof} The Weitzenb\"ock formula states that
$-\Delta_1/2=\nabla^{\dagger}_1\nabla_1/2+ \IRR/2$, 
where $\nabla_1$ stands for the Levi-Civita connection acting on $1$-forms. Under
the given assumptions 
on $\IRR$, we can use Theorem 2.13 in \cite{G1} to define the form sum $\tilde{H_1}$
of the Friedrichs 
realization of $\nabla^{\dagger}_1\nabla_1/2$ and the multiplication operator
$\IRR/2$ in 
$\Omega^k_{\mathsf{L}^{2}}(M)$. But the geodesic completeness assumption implies the
essential 
self-adjointness of $-\Delta_1$ on the domain of definition
$\Omega^1_{\mathsf{C}^{\infty}_0}(M)$ 
(this essential self-adjointness is a classical result \cite{str}; under our
assumption on $\IRR$, 
this also follows from the main result of \cite{post}), and as a consequence we get
$H_1=\tilde{H_1}$. \\
We define scalar potentials $w_j:M\to [0,\infty)$, $w:M\to \IR$ by
  \begin{align}
    w_1:= \min\sigma(\IRR_1/2(\bullet)),\>w_2:=
    \max\sigma(\IRR_2/2(\bullet)),\>w:=w_1-w_2,\nn
  \end{align}
where $\sigma(\IRR_j/2(x))$ stands for the spectrum of the nonnegative self-adjoint
operator 
$\IRR_j/2(x):\mathrm{T}^*_xM \to \mathrm{T}^*_xM$. Then clearly
$w_1\in\mathsf{L}^1_{\mathrm{loc}}(M)$, 
$w_2\in \mathcal{K}(M)$, and the above considerations combined with Theorem 2.13
from \cite{G2} 
(semigroup domination), Theorem 2.9 from \cite{G2} (a scalar Feynman-Kac formula)
and $-w\leq w_2$ 
imply the first inequality in  
\begin{align}
\left|\mathrm{e}^{-t H_1}\alpha(x)\right|_x&\leq 
\mathbb{E}\left[\mathrm{e}^{\int^t_0 w_2(B_s(x))\Id
s}|\alpha|(B_t(x))1_{\{t<\zeta(x)\}}\right]\nn\\
&\leq \left\|\alpha\right\|_{\infty}\mathbb{E}
\left[\mathrm{e}^{\int^t_0 w_2(B_s(x))\Id s}1_{\{t<\zeta(x)\}}\right]\>\text{ for
a.e. $x\in M$}.
\end{align}
But now the assertion follows readily from Lemma \ref{xdd}.
\end{proof}

Now we can prove Theorem \ref{BV}:

\begin{proof}[Proof of Theorem \ref{BV}] Firstly, the inequality
\[
\mathrm{Var}(f)\leq \liminf_{t\to 0+} \int_M\left|\Id \mathrm{e}^{-tH}f(x)\right|_x 
\mathrm{vol}(\Id x)
\]
can be deduced exactly as in the proof of Theorem 3.2 in \cite{carbo}. In fact, this
inequality is always 
satisfied without any assumptions on the Riemannian structure of $M$. To see this,
just note that if 
$\alpha\in\Omega^1_{\mathsf{C}^{\infty}_0}(M)$ is such that
$\left\|\alpha\right\|_{\infty}\leq 1$, then 
\begin{align}
\left|\int_M  \overline{ f(x)}\Id^{\dagger} \alpha (x) \mathrm{vol}(\Id x)\right|&=
\left|\lim_{t\to 0+}\int_M \overline{ \mathrm{e}^{-tH}f(x)} \Id^{\dagger} \alpha(x) 
\mathrm{vol}(\Id x)\right|\nn\\
&=\left|\lim_{t\to 0+}\int_M (\Id\mathrm{e}^{-tH}f(x) , \alpha(x))_x 
\mathrm{vol}(\Id x)\right|\nn\\
&\leq \liminf_{t\to 0+} \int_M\left|\Id \mathrm{e}^{-tH}f(x)\right|_x
\mathrm{vol}(\Id x),
\end{align}
where we have used $\left\| \mathrm{e}^{-tH}f-f\right\|_1\to 0$ as $t\to 0+$.\\
In order to prove 
\begin{align}
\mathrm{Var}(f)\geq \limsup_{t\to 0+} \int_M\left|\Id \mathrm{e}^{-tH}f(x)\right|_x
\mathrm{vol}(\Id x),
\label{d13}
\end{align}
we first remark the well-known fact (see for example the appendix of \cite{thal})
that geodesic 
completeness implies
\begin{align}
\mathrm{e}^{-t H_1}\Id h=\Id \mathrm{e}^{-t H}h\>\>\text{ for any
$h\in\mathsf{C}^{\infty}_0(M)$}.\label{h1}
\end{align}
Now let $\alpha\in\Omega^1_{\mathsf{C}^{\infty}_0}(M)$ be such that
$\left\|\alpha\right\|_{\infty}\leq 1$, 
and let $t,\epsilon>0$ be arbitrary. Then Lemma \ref{smoo} shows
\begin{align}
\left\|\mathrm{e}^{-t H_1}\alpha\right\|_{\infty}\leq (1+\epsilon)
\mathrm{e}^{tC(\epsilon)}, \label{h2}
\end{align}
and applying (\ref{h1}) with $h=\Id^{\dagger}\alpha$ gives (by testing against 
$\tilde{h}\in \mathsf{C}^{\infty}_0(M)$) the identity 
$\Id^{\dagger} \mathrm{e}^{-t H_1}\alpha=\mathrm{e}^{-t H}\Id^{\dagger}\alpha$. So
using 
$\mathrm{e}^{-tH}\in \ILL(\mathsf{L}^{\infty}(M))$, we can conclude 
$\mathrm{e}^{-t H_1}\alpha\in\Omega^1_{\mathrm{bd}}(M)$. Finally, combining 
$\mathrm{e}^{-t H_1}\alpha\in\Omega^1_{\mathrm{bd}}(M)$, (\ref{h2}), Lemma \ref{d12}
and 
$\Id^{\dagger} \mathrm{e}^{-t H_1}\alpha=\mathrm{e}^{-t H}\Id^{\dagger}\alpha$, we have
\begin{align}
\left|\int_M \overline{\mathrm{e}^{-t H}f(x)}\Id^{\dagger}\alpha(x) 
\mathrm{vol}(\Id x)\right|&=
\left|\int_M \overline{f(x)}\Id^{\dagger} \mathrm{e}^{-t H_1}\alpha(x)
\mathrm{vol}(\Id x)\right|\nn\\
&\leq (1+ \epsilon) \mathrm{e}^{tC(\epsilon)}\mathrm{Var}(f), 
\end{align}
so that taking the supremum over all such $\alpha$ and using Proposition \ref{d14}
we arrive at
\begin{align}
 \int_M\left|\Id \mathrm{e}^{-tH}f(x)\right|_x \mathrm{vol}(\Id x)&=
\sup_{\alpha}\left|\int_M \overline{\mathrm{e}^{-t
H}f(x)}\Id^{\dagger}\alpha(x)\mathrm{vol}(\Id x)\right|\nn\\
 & \leq (1+ \epsilon) \mathrm{e}^{tC(\epsilon)} \mathrm{Var}(f),
\end{align}
which proves (\ref{d13}) by first taking $\limsup_{t\to 0+}$, and then taking
$\lim_{\epsilon\to 0+}$.
\end{proof} 
 
\appendix
\section{Vector measures on locally compact spaces }\label{complexmeas}

Let $\IK$ be either $\IC$ or $\IR$, and denote the corresponding standard
inner-product and norm on $\IK^m$ with $(\bullet,\bullet)_{\IK^m}$ and
$\left|\bullet\right|_{\IK^m}$. Let $X$ be a locally compact Hausdorff space with
its Borel-$\sigma$-algebra ${\mathcal B}(X)$. We denote with
$\mathsf{C}_{\infty}(X,\IK^m)$ the space of $\IK^m$-valued functions on $X$ that
vanish at infinity, which is a $\IK$-Banach space with respect to the uniform norm
$\left\|\bullet\right\|_{\infty}$.\\
A \emph{$\IK^m$-valued Borel measure} on $X$ is defined to be a countably additive
set function $\nu:{\mathcal B}(X)\to\IK^m$, and its \emph{total variation} measure
is the positive finite Borel measure defined for $B\in \IB(X)$ by 
\begin{equation}\label{e1}
|\nu|(B)=\sup\left\{\left. \sum_{j=1}^\infty \left|\nu(B_j)\right|_{\IK^m}\right|\>
B_j\in \IB(X) \> \text{ for all $j\in\IN$}, B=\bigsqcup_{j=1}^\infty B_j\right\}.
\nn
\end{equation}
Of course one has $\nu=|\nu|$, in case $\nu$ itself is a positive finite Borel
measure. \\
Let us denote the $\IK$-linear space of $\IK^m$-valued Borel measure on $X$ with
$\tilde{\IMM}(X,\IK^m)$. Then one has: 

\begin{Proposition}\label{rm} 
\emph{a)} $\tilde{\IMM}(X,\IK^m)$ is a $\IK$-Banach space with respect to the total
variation norm, and the map
\[
\tilde{\Psi}: \tilde{\IMM}(X,\IK^m)\longrightarrow 
(\mathsf{C}_{\infty}(X,\IK^m))^*,\>\> \tilde{\Psi}(\nu)(f)=
\int_X (f,\Id \nu)_{\IK^m} 
\]
is an isometric isomorphism of $\IK$-linear spaces. If $m=1$, then $\tilde{\Psi}$ is
order preserving.\\
\emph{b)} The map
\[
\tilde{\IMM}(X,\IC^m)\longrightarrow \tilde{\IMM}(X,\IR^{2m}),\>\>\nu\longmapsto 
\nu_{\IR^{2m}}:=(\mathrm{Re}(\nu),\mathrm{Im}(\nu))
\]
is an isometric isomorphism of $\IR$-linear spaces.
\end{Proposition}

\begin{proof} Part a) is just a variant of Riesz-Markoff\rq{}s theorem, see e.g.
\cite[Theorem 6.19]{Rudin}. 
In part b) it is clear that the indicated map is an $\IR$-linear isomorphism. The
essential observation 
for the proof of $|\nu|=|\nu_{\IR^{2m}}|$ is the polar decomposition: Namely, by a
Radon-Nikodym-type argument one gets the existence of a Borel function
$\sigma:X\to\IC^m$, with $|\sigma|=1$ $|\nu|$-a.e. in $X$, 
such that $\Id \nu=\sigma \ \Id |\nu|$. As the norms $|\bullet|_{\IR^{2m}}$ and
$|\bullet|_{\IC^{m}}$ are equal under 
the canonical identification of $\IR^{2m}$ with $\IC^{m}$, it follows directly from
the definition of the 
total variation measure that $|\nu|=|\nu_{\IR^{2m}}|$. 
\end{proof}


\begin{thebibliography}{99}

\bibitem{ai} 
Aizenman, M. \& Simon, B.: 
\emph{ Brownian motion and Harnack inequality for Schr\"odinger operators.} 
Comm. Pure Appl. Math. 35 (1982), no. 2, 209--273.

\bibitem{ambro} 
Ambrosio, L. \&  Fusco, N. \&  Pallara, D.: 
\emph{ Functions of bounded variation and free discontinuity problems.} 
Oxford Mathematical Monographs. The Clarendon Press, Oxford University Press, New
York, 2000.


\bibitem{BVWiener1}
Ambrosio, L. \& Maniglia, S. \& Miranda, M. \& Pallara, D.:
\emph{$\mathsf{BV}$ functions in abstract Wiener spaces},
J. Funct. Anal. 258 (2010), 785--813. 

\bibitem{aubin} Aubin, T.: \emph{ Espaces de Sobolev sur les variétés riemanniennes.} Bull. Sci. Math. (2) 100 (1976), no. 2, 149--173. 


\bibitem{baldi} 
Baldi, A.: \emph{ Weighted BV functions.}  Houston J. Math. 27 (2001), no. 3,
683--705. 

\bibitem{Bakry}
Bakry, D.: 
\emph{ \'Etude des transformations de Riesz dans les vari\'et\'es riemanniennes \`a
courbure de Ricci minor\'ee. } 
S\'eminaire de Probabilit\'es, XXI, 137--172, Lecture Notes in Math., 1247,
Springer, Berlin, 1987.

\bibitem{besse} Besse, A. L.: \emph{ Einstein manifolds.}  Springer-Verlag, Berlin,
1987.

\bibitem{Bro1} Broderix, K. \& Hundertmark, D. \& Leschke, H.: {\it Continuity
properties of Schr\"odinger 
semigroups with magnetic fields.} Reviews in Mathematical Physics 12 (2000), 181--225.

\bibitem{Carnot}
Bramanti, M. \& Miranda, M. \& Pallara, D.:
\emph{Two characterization of $\mathsf{BV}$ functions on Carnot groups via the heat
semigroup},
Int. Math. Res. Not., 17 (2012), 3845-3876. 

\bibitem{carbo} 
Carbonaro, A. \& Mauceri, G.: 
\emph{ A note on bounded variation and heat semigroup on Riemannian manifolds. } 
Bull. Austral. Math. Soc. 76 (2007), no. 1, 155--160. 

\bibitem{coul} Coulhon, T. \& Zhang, Qi S.: \emph{ Large time behavior of heat
kernels on forms.} 
J. Differential Geom. 77 (2007), no. 3, 353--384.

\bibitem{giorgi} 
De Giorgi, E.: 
\emph{Su una teoria generale della misura $(r-1)$-dimensionale in uno spazio ad $r$
dimensioni.}  
Ann. Mat. Pura Appl. 36 (1954), 191--213. Also in: Ennio De~Giorgi: \emph{ Selected
Papers}, 
(L. Ambrosio, G. Dal Maso, M. Forti, M. Miranda, S. Spagnolo eds.), Springer, 2006,
79--99. 
English translation, {\em Ibid.}, 58--78.

\bibitem{dem} 
Demuth, M. \& van Casteren, J.A.: 
\emph{  Stochastic spectral theory for selfadjoint Feller operators. A functional
integration approach.} 
Probability and its Applications. Birkh\"auser Verlag, Basel, 2000. 

\bibitem{dod} 
Dodziuk, J. \& Mathai, V.: 
\emph{ Kato\rq{}s inequality and asymptotic spectral properties for discrete
magnetic Laplacians.} 
Contemp. Math., vol. 398, Amer. Math. Soc., Providence, 2006, pp. 69--81.

\bibitem{thal} 
Driver, B.K. \& Thalmaier, A.:  \emph{ Heat equation derivative formulas for vector
bundles.} 
J. Funct. Anal. 183 (2001), no. 1, 42--108.

\bibitem{FH}
Fukushima, M. \& Hino, M.: 
\emph{On the space of BV functions and a Related Stochastic Calculus in Infinite
Dimensions.} 
J. Funct. Anal., 183 (2001), 245--268.

\bibitem{grig} 
Grigor'yan, A.: {\it Heat kernel and analysis on manifolds.} 
AMS/IP Studies in Advanced Mathematics, 47. 
American Mathematical Society, Providence, RI; International Press, Boston, MA, 2009.

\bibitem{G1} 
G\"uneysu, B.: 
\emph{Kato's inequality and form boundedness of Kato potentials on arbitrary
Riemannian manifolds.} 
To appear in Proc. Amer. Math. Soc. (2012).

\bibitem{post} 
G\"uneysu, B. \& Post. O.: 
\emph{Path integrals and the essential self-adjointness of differential operators on
noncompact manifolds.} To appear in Math.Z.


\bibitem{G2} 
G\"uneysu, B.: 
\emph{On generalized Schr\"odinger semigroups.}  J. Funct. Anal. 262 (2012),
4639--4674.

\bibitem{G6} G\"uneysu, B.: \emph{Nonrelativistic Hydrogen type stability problems
on nonparabolic 3-manifolds}. 
Ann. Henri Poincar{\'e} 2012, Volume 13, Issue 7,  1557--1573.

\bibitem{hebey} Hebey, E.: 
\emph{ Sobolev spaces on Riemannian manifolds.} 
Lecture Notes in Mathematics, 1635. Springer-Verlag, Berlin, 1996.

\bibitem{Hsu} 
Hsu, E.: \emph{Stochastic Analysis on Manifolds.} Graduate Studies in Mathematics, 38. 
American Mathematical Society, Providence, RI, 2002.

\bibitem{kato2} Kato, T.: {\it Schr\"odinger operators with singular potentials.}
Israel J. Math. 13 (1972).

\bibitem{kt} 
Kuwae, K. \& Takahashi, M.: 
\emph{Kato class measures of symmetric Markov processes under heat kernel estimates.} 
J. Funct. Anal. 250 (2007), no. 1, 86--113.

\bibitem{meyers} 
Meyers, N.G. \& Serrin, J.: \emph{H=W}. Proc. Nat. Acad. Sci. U.S.A. 51 (1964),
1055--1056. 

\bibitem{ognjen} 
Milatovic, O.: \emph{Essential self-adjointness of magnetic Schr\"odinger operators
on locally finite graphs} 
Integral Equations and Operator Theory, Volume 71, Number 1, (2011), 13--27.

\bibitem{mira} 
Miranda, M., Jr. \& Pallara, D. \& Paronetto, F. \& Preunkert, M.: 
\emph{Heat semigroup and functions of bounded variation on Riemannian manifolds.}  
J. Reine Angew. Math. 613 (2007), 99--119.

\bibitem{Rudin}
Rudin, W.:
\emph{ Real and complex analysis.} 
McGraw-Hill, Singapore, 1987. 

\bibitem{sal} Saloff-Coste, L.: \emph{
Aspects of Sobolev-type inequalities.} London Mathematical Society Lecture Note
Series, 289. 
Cambridge University Press, Cambridge, 2002.

\bibitem{shub} 
Shubin, M.: \emph{ Essential self-adjointness for semi-bounded magnetic
Schr\"odinger operators on non-compact 
manifolds.} J. Funct. Anal. 186 (2001), no. 1, 92--116.

\bibitem{sg} Simon, B.: \emph{ Schr\"odinger semigroups.} Bull. Amer. Math. Soc.
(N.S.) 7 (1982), no. 3, 447--526.

\bibitem{str} 
Strichartz, R. S.: \emph{ Analysis of the Laplacian on the complete Riemannian
manifold.} 
J. Funct. Anal. 52 (1983), no. 1, 48--79.

\bibitem{sturm} Sturm, K.-T.:\emph{ Schrödinger semigroups on manifolds.} J. Funct.
Anal. 118 (1993), no. 2, 309--350.

\bibitem{sznit} 
Sznitman, A.-S.: \emph{ Brownian motion, obstacles and random media.} 
Springer Monographs in Mathematics. Springer-Verlag, Berlin, 1998.

\end{thebibliography}
\end{document}